\documentclass[11pt,a4paper]{amsart}

\usepackage{latexsym}
\usepackage{amsmath}
\usepackage{amsthm}
\usepackage{amssymb}
\usepackage{vmargin}
\usepackage{amscd}
\usepackage{stmaryrd}
\usepackage{euscript}
\usepackage{mathrsfs}
\usepackage{amscd}
\usepackage[all]{xy}
\usepackage{xr}

\DeclareMathAlphabet{\mathpzc}{OT1}{pzc}{m}{it}

\newtheorem{theorem}{Theorem}[section]
\newtheorem{proposition}[theorem]{Proposition}
\newtheorem{corollary}[theorem]{Corollary}

\newtheorem{lemma}[theorem]{Lemma}
\newtheorem*{theorem*}{Theorem}
\newtheorem*{proposition*}{Proposition}
\newtheorem*{corollary*}{Corollary}
\newtheorem*{lemma*}{Lemma}
\newtheorem*{conjecture*}{Conjecture}

\theoremstyle{definition}
\newtheorem{definition}[theorem]{Definition}

\newtheorem*{definition*}{Definition}

\theoremstyle{remark}
\newtheorem{example}[theorem]{Example}
\newtheorem{examples}[theorem]{Examples}
\newtheorem{remark}[theorem]{Remark}
\newtheorem{remarks}[theorem]{Remarks}

\newtheorem*{example*}{Example}
\newtheorem*{examples*}{Examples}
\newtheorem*{remark*}{Remark}
\newtheorem*{remarks*}{Remarks}
\newtheorem*{exercise*}{Exercise}

\newcommand\ra{\rightarrow}

\newcommand\id{\mathrm{id}}

\newcommand\ten{\otimes}
\newcommand\vareps{\varepsilon}
\newcommand\eps{\epsilon}

\newcommand\CC{\mathrm{C}}

\renewcommand\H{\mathrm{H}}

\newcommand\N{\mathbb{N}}
\newcommand\Z{\mathbb{Z}}

\newcommand\bD{\mathbb{D}}
\newcommand\bE{\mathbb{E}}

\newcommand\bI{\mathbb{I}}
\newcommand\bJ{\mathbb{J}}

\newcommand\bS{\mathbb{S}}

\newcommand\C{\mathcal{C}}

\newcommand\cA{\mathcal{A}}
\newcommand\cB{\mathcal{B}}

\newcommand\cD{\mathcal{D}}
\newcommand\cE{\mathcal{E}}

\newcommand\cS{\mathcal{S}}

\newcommand\cU{\mathcal{U}}

\renewcommand\O{\mathscr{O}}

\newcommand\Def{\mathfrak{Def}}

\newcommand\fY{\mathfrak{Y}}
\newcommand\fZ{\mathfrak{Z}}

\renewcommand\L{\Lambda}

\newcommand\m{\mathfrak{m}}

\newcommand\g{\mathfrak{g}}

\renewcommand\hom{\mathscr{H}\!\mathit{om}}
\newcommand\cHom{\mathcal{H}\!\mathit{om}}

\newcommand\Ho{\mathrm{Ho}}
\newcommand\Ring{\mathrm{Ring}}

\newcommand\Hom{\mathrm{Hom}}

\newcommand\Ext{\mathrm{Ext}}
\newcommand\EExt{\mathbb{E}\mathrm{xt}}

\newcommand\cone{\mathrm{cone}}

\newcommand\Ob{\mathrm{Ob}\,}

\newcommand\Gp{\mathrm{Gp}}

\newcommand\Spf{\mathrm{Spf}\,}

\newcommand\Set{\mathrm{Set}}

\newcommand\Cat{\mathrm{Cat}}

\newcommand\Mor{\mathrm{Mor}\,}

\newcommand\Sp{\mathrm{Sp}}

\newcommand\Grpd{\mathrm{Grpd}}

\newcommand\into{\hookrightarrow}
\newcommand\onto{\twoheadrightarrow}
\newcommand\abuts{\implies}
\newcommand\xra{\xrightarrow}

\newcommand\bt{\bullet}
\newcommand\by{\times}

\newcommand\mc{\mathrm{MC}}
\newcommand\mmc{\underline{\mathrm{MC}}}

\newcommand\Gg{\mathrm{Gg}}

\newcommand\ddef{\mathrm{Def}}
\newcommand\ddel{\mathrm{Del}}

\newcommand\Tot{\mathrm{Tot}\,}

\newcommand\toph{\top_{\mathrm{h}}}
\newcommand\topv{\top_{\mathrm{v}}}
\newcommand\both{\bot_{\mathrm{h}}}
\newcommand\botv{\bot_{\mathrm{v}}}

\newcommand\ev{\mathrm{ev}}

\newcommand\pd{\partial}

\newcommand\half{\frac{1}{2}}

\newcommand\LA{\mathrm{LA}}

\newcommand\op{\mathrm{opp}}

\newcommand\oR{\mathbf{R}}
\newcommand\oL{\mathbf{L}}

\newcommand\uleft\underleftarrow
\newcommand\uline\underline

\begin{document}

\begin{abstract}
We introduce a new approach to constructing derived deformation groupoids, by considering them as parameter spaces for strong homotopy bialgebras. This allows them to be constructed for all classical deformation problems, such as deformations of an arbitrary scheme,  in any characteristic.
\end{abstract}

\title{Derived deformations of schemes}
\author{J.P.Pridham}
\thanks{This work was supported by Trinity College, Cambridge; and by the Engineering and Physical Sciences Research Council [grant number  EP/F043570/1.}
\maketitle


\section*{Introduction}

In \cite{paper2}, the theory of simplicial deformation complexes (SDCs) was expounded  as a means of governing deformation problems, giving an alternative to the theory of differential graded Lie algebras (DGLAs). The main advantages of SDCs over DGLAs are that they can be constructed canonically (and thus for a wider range of problems), and are valid in all characteristics. 

In \cite{Man2}, Manetti showed that given a DGLA, or even  an SHLA, governing a deformation problem, it is possible  to define an extended deformation functor. The approach in this paper can almost be regarded as opposite to this --- we try, for any deformation problem, to define an extended deformation functor with a geometric interpretation, meaning that the functor still parametrises geometric objects. We then see how this functor can be recovered from   the SDC governing  the problem.

Since almost all examples of SDCs come from monadic and comonadic adjunctions, in Section \ref{extsdc} we  look at how to extend deformation groupoids in these scenarios. For a monad $\top$, the solution is to look at the strong homotopy $\top$-algebras, as defined by Lada in \cite{loop}. The idea is that the monadic axioms are only satisfied up to homotopy, with the homotopies satisfying further conditions up to homotopy, and so on. This approach allows us to define a quasi-smooth extended deformation functor associated to any SDC, with the same cohomology.

Using the constructions of \S\S \ref{diagram} and \ref{constrain}, we  describe extended deformations of morphisms and diagrams (giving new results  even for the problems in \cite{paper2}). This  defines cohomology of a morphism in any such category, giving a variant of Van Osdol's bicohomology (\cite{osdol}). One consequence is that the space describing extended deformations of the identity morphism on an object $D$ is just the loop space of the space of extended deformations of $D$. 

The structure of the paper is as follows.  Sections \ref{one} and \ref{algsn} are introductory, summarising results from \cite{ddt1} and properties of monads and comonads, respectively. Section \ref{constructsdc}  reprises material from \cite{paper2} on SDCs, and includes  new results constructing SDCs associated to  diagrams in \S\S \ref{diagram} and \ref{constrain}. The key motivating examples of deformations of a scheme are described in Examples \ref{keyegs} and \ref{overscheme}. 

Section \ref{extsdc} then gives the construction of the derived deformation functor (Definition \ref{ddefdef}), together with a simplified description of derived deformations of a morphism (Proposition \ref{loopmor}), and the characterisation of derived deformations of an identity morphism as a loop space (Proposition \ref{loopmor2}).

In \cite{paper2}, it was shown that SDCs are equivalent to $\N_0$-graded  DGLAs in characteristic $0$, in such a way  that the associated deformation groupoids are equivalent. In Appendix \ref{ddtconsistentsdc}, we show how that the associated extended deformation functors are also equivalent.

\tableofcontents

\section{Derived deformation functors}\label{one}

With the exception of \S \ref{quotsp},   the definitions and results in this section can all be found in \cite{ddt1}. 
Fix a complete local Noetherian ring $\L$, with maximal ideal $\mu$ and residue field $k$. 

\subsection{Simplicial Artinian rings}

\begin{definition}
Let $\C_{\L}$ denote the category of local Artinian $\L$-algebras with residue field $k$.
 We define $s\C_{\L}$ to be the category of Artinian simplicial  local $\L$-algebras, with residue field $k$. 
\end{definition}

\begin{definition}\label{N^s}
Given a simplicial complex $V_{\bt}$, recall that the normalised chain complex $N^s(V)_{\bt}$ is given by  $N^s(V)_n:=\bigcap_{i>0}\ker (\pd_i: V_n \to V_{n-1})$, with differential $\pd_0$. The simplicial Dold-Kan correspondence says that $N^s$ gives an equivalence of categories between simplicial complexes and non-negatively graded chain complexes in any abelian category. Where no ambiguity results, we will denote $N^s$ by $N$.
\end{definition}

\begin{lemma}\label{cotdef}
A simplicial complex $A_{\bt}$  of local $\L$-algebras with residue field $k$ and maximal ideal $\m(A)_{\bt}$   is Artinian if and only if:
\begin{enumerate}
\item the normalisation $N(\cot A)$ of the cotangent space $\cot A:=\m(A)/(\m(A)^2+\mu \m(A))$  is finite-dimensional (i.e. concentrated in finitely many degrees, and finite-dimensional in each degree). 
\item For some $n>0$, $\m(A)^n=0$.
\end{enumerate} 
\end{lemma}
\begin{proof}
\cite{ddt1} Lemma \ref{ddt1-cotdef}
\end{proof}

 As in \cite{descent},  we say that a functor is left exact if it preserves all finite limits. This is equivalent to saying that it preserves final objects and fibre products.

\begin{definition}\label{spdef}
Define $\Sp$ to be  the category  of left-exact functors from $\C_{\L}$ to $\Set$. 
Define  $c\Sp$ to be the category of left-exact functors from $s\C_{\L}$ to $\Set$.
\end{definition}

\begin{definition}
Given a functor $F:\C_{\L} \to \Set$, we write $F:s\C_{\L} \to \Set$ to mean $A \mapsto F(A_0)$  (corresponding to the inclusion $\Sp \into c\Sp$).
\end{definition}

\subsection{Properties of morphisms}

\begin{definition}\label{smoothdef}
As in \cite{Man}, we say that a functor $F:\C_{\L}\to \Set$ is smooth if for all surjections $A \to B$ in $\C_{\L}$, the map $F(A) \to F(B)$ is surjective. 
\end{definition}

\begin{definition}
We say that a map $f:A \to B$ in $s\hat{\C}_{\L}$ is acyclic if $\pi_i(f):\pi_i(A) \to \pi_i(B)$ is an isomorphism of pro-Artinian $\L$-modules for all $i$.  $f$ is said to be surjective if each $f_n:A_n \to B_n$ is  surjective.
\end{definition}

Note that for any simplicial abelian group $A$, the homotopy groups can be calculated by $\pi_iA \cong \H_i(NA)$, the homology groups of the normalised chain complex. These in turn are isomorphic to the homology groups of the unnormalised chain complex associated to $A$. 

\begin{definition}
We define a small extension $e:I \to A \to B$ in $s\C_{\L}$ to consist of a surjection $A \to B$ in $s\C_{\L}$ with kernel $I$, such that $\m(A)\cdot I=0$. Note that this implies that $I$ is a simplicial complex of $k$-vector spaces.
\end{definition}

\begin{lemma}\label{small}
Every surjection in $s\C_{\L}$ can be factorised as a composition of small extensions. Every acyclic surjection in $s\C_{\L}$ can be factorised as a composition of acyclic small extensions.  
\end{lemma}
\begin{proof}
\cite{ddt1} Lemma \ref{ddt1-small}.
\end{proof}

\begin{definition}
We say that a morphism $\alpha:F\to G$ in $c\Sp$  is smooth if for all small extensions $A \onto B$ in $s\C_{\L}$, the map $F(A) \to F(B)\by_{G(B)}G(A)$ is surjective. 

Similarly, we call $\alpha$ quasi-smooth if for all acyclic small extensions $A \to B$ in $s\C_{\L}$, the map $F(A) \to F(B)\by_{G(B)}G(A)$ is surjective.
\end{definition}

\begin{lemma}\label{ctod} A morphism  $\alpha:F\to G$ in $\Sp$ is smooth if and only if the induced morphism between the  objects $F, G \in c\Sp$ is quasi-smooth, if and only if it is smooth.
\end{lemma}
\begin{proof}
\cite{ddt1} Lemma \ref{ddt1-ctod}.
\end{proof}

\subsection{Derived deformation functors}

\begin{definition}
Define the  $sc\Sp$ to be the category of left-exact functors from $s\C_{\L}$ to the category $\bS$ of simplicial sets.
\end{definition}
\begin{definition}\label{scspqsdef}
 A morphism $\alpha:F\to G$ in    $sc\Sp $ is  said to be smooth if 
\begin{enumerate}
\item[(S1)]
for every acyclic surjection $A \to B$ in $s\C_{\L}$, the map $F(A)\to F(B)\by_{G(B)}G(A)$ is a trivial fibration in $\bS$; 
\item[(S2)]
for every surjection $A \to B$ in $s\C_{\L}$, the map $F(A)\to F(B)\by_{G(B)}G(A)$ is a surjective fibration in $\bS$.
\end{enumerate}

A morphism $\alpha:F\to G$ in    $sc\Sp $  is  said to be quasi-smooth if it satisfies (S1) and
\begin{enumerate}
\item[(Q2)]
for every surjection $A \to B$ in $s\C_{\L}$, the map $F(A)\to F(B)\by_{G(B)}G(A)$ is a  fibration in $\bS$.
\end{enumerate}
\end{definition}

\begin{definition}
Given $A \in s\C_{\L}$ and a finite simplicial set $K$, define $A^K \in \C_{\L}$ by 
$$
(A^K)_i:=\Hom_{\bS}(K\by \Delta^i, A)\by_{\Hom_{\Set}(\pi_0K, k)}k.
$$
\end{definition}

\begin{definition}\label{underline}
Given   $F \in sc\Sp$, define $\underline{F}:s\C_{\L}\to \bS$ by 
$$
\underline{F}(A)_n:= F_n(A^{\Delta^n}).
$$
 
For $F \in c\Sp$, we may regard $F$ as an object of $sc\Sp$ (with the constant simplicial structure), and then define $\underline{F}$  as above.
\end{definition}

\begin{lemma}\label{settotop} A map $\alpha:F\to G $ in $c\Sp$ is smooth (resp. quasi-smooth) if and only if the induced map of functors  $\underline{\alpha}:\underline{F}\to \underline{G}$  is smooth (resp. quasi-smooth) in $sc\Sp$.
\end{lemma}
\begin{proof}
\cite{ddt1} Lemma \ref{ddt1-settotop}.
\end{proof}

The following Lemma will provide many examples of functors which are quasi-smooth but not smooth. 
\begin{lemma}\label{sm7}
If $F\to G$ is a quasi-smooth map of functors $F,G:s\C_{\L} \to \bS$, and $K \to L$ is a cofibration in $\bS$, then 
$$
F^L \to F^K\by_{G^K}G^L
$$
is quasi-smooth.
\end{lemma}
\begin{proof}
This is an immediate consequence of the fact that $\bS$ is a simplicial model category, following from axiom SM7, as given in \cite{sht} \S II.3.
\end{proof}

The following lemma is a consequence of  standard properties of fibrations and trivial fibrations in $\bS$.
\begin{lemma}\label{basechange}
If $F\to G$ is a quasi-smooth map of functors $F,G:s\C_{\L} \to \bS$, and $H \to G$ is any map of functors, then $F\by_GH \to H$ is quasi-smooth. 
\end{lemma}

\begin{definition}\label{pioqs}
A map $ \alpha:F\to G$ of functors   $F,G:\C_{\L}\to \bS$ is said to be smooth (resp. quasi-smooth, resp. trivially smooth)  if for all surjections $A \onto B$ in $\C_{\L}$, the maps 
$$
 F(A) \to F(B)\by_{G(B)}G(A)
$$ 
  are surjective fibrations (resp. fibrations, resp. trivial fibrations).
\end{definition}

\begin{proposition}\label{smoothchar}
A map  $ \alpha:F\to G$  of left-exact functors   $F,G:\C_{\L}\to \bS$ is smooth if and only if the maps  $F_n\xra{\alpha_n} G_n$ of functors   $F_n,G_n:\C_{\L}\to \Set$ are all smooth.
\end{proposition}
\begin{proof}
\cite{ddt1} Proposition \ref{ddt1-smoothchar}.
\end{proof}

\begin{proposition}\label{ctodtnew}
If a morphism $F\xra{\alpha} G$ of left-exact functors $F,G:s\C_{\L} \to \bS$ is such that the maps
$$
\theta: F(A)\to F(B)\by_{G(B)}G(A)
$$
 are surjective fibrations  for all acyclic small extensions $A \to B$, then $\underline{\alpha}:\underline{F} \to \underline{G}$ is quasi-smooth (resp. smooth) if and only if  $\theta$ is a fibration (resp. surjective fibration) for all small extensions $A \to B$.
\end{proposition}
\begin{proof}
\cite{ddt1} Proposition \ref{ddt1-ctodtnew}.
\end{proof}

\begin{definition}\label{weakdef}
We will say that a morphism $\alpha: F \to G$ of quasi-smooth objects of $sc\Sp$ is a weak equivalence if, for all $A \in s\C_{\L}$, the maps $\pi_i F(A) \to \pi_iG(A)$ are isomorphisms for all $i$.
\end{definition}

\subsection{Quotient spaces}\label{quotsp}

\begin{definition}
Given  functors $X:s\C_{\L} \to \bS$ and $G: s\C_{\L} \to s\Gp$, together with a right action of $G$ on $X$, define the quotient space by
$$
[X/G]_n= (X\by^{G}WG)_n=  X_n\by G_{n-1}\by G_{n-2}\by \ldots G_0,
$$
with operations as standard for universal bundles (see \cite{sht} Ch. V). Explicitly:
\begin{eqnarray*}
\pd_i(x,g_{n-1},g_{n-2}, \ldots,g_0)&=& \left\{ \begin{matrix}(\pd_0x*g_{n-1}, g_{n-2}, \ldots, g_0)& i=0;\\
 (\pd_ix, \pd_{i-1}g_{n-1}, \ldots, (\pd_0g_{n-i})g_{n-i-1}, g_{n-i-2}, \ldots, g_0) & 0<i<n;\\
(\pd_nx, \pd_{n-1}g_{n-1},\ldots, \pd_1g_1) & i=n;
\end{matrix}\right.\\
\sigma_i(x,g_{n-1},g_{n-2}, \ldots,g_0)&=& (\sigma_ix, \sigma_{i-1}g_{n-1}, \ldots, \sigma_0g_{n-i},e,g_{n-i-1}, g_{n-i-2}, \ldots, g_0).
\end{eqnarray*}
The space $[\bullet/G]$ is also denoted $\bar{W}G$, and is a model for the classifying space $BG$ of $G$.  Note replacing $WG$ with  any other fibrant cofibrant contractible $G$-space $EG$ will give the same properties.

\end{definition}

\begin{lemma}
If $G:s\C_{\L} \to s\Gp$ is smooth, then  $\bar{W}G$ is smooth.
\end{lemma}
\begin{proof}
For any surjection $A \to B$, we have $G(A) \to G(B)$ fibrant and surjective on $\pi_0$, which by \cite{sht} Corollary V.6.9 implies that $\bar{W}G(A) \to \bar{W}G(B)$ is a fibration. If $A \to B$ is also acyclic, then everything is trivial by properties of $\bar{W}$ and $G$.
\end{proof}

\begin{remark}
Observe that this is our first example of a quasi-smooth functor  which is not a right  Quillen functor for the simplicial model structure. The definitions of smoothness and quasi-smoothness  were designed with $\bar{W}G$ in mind.
\end{remark}

\begin{lemma}
If $X$ is quasi-smooth, then so is $[X/G] \to \bar{W}G$.
\end{lemma}
\begin{proof}
This follows from the observation that for any fibration (resp. trivial fibration) $Z \to Y$ of $G$-spaces, $[Z/G]\to [Y/G]$ is a fibration (resp. trivial fibration).
\end{proof}

\begin{corollary}\label{fibquot}
If $X$ is quasi-smooth and $G$  smooth, then  $[X/G]$ is quasi-smooth.
\end{corollary}
\begin{proof}
Consider the fibration $X \to [X/G] \to \bar{W}G$.
\end{proof}

\subsection{Cohomology and obstructions}\label{cohomology}

Given a quasi-smooth morphism  $\alpha:F\to G$  in $sc\Sp$, there exist $k$-vector spaces  $\H^i(F/G)$ for all $i \in \Z$. 

By \cite{ddt1} Corollary \ref{ddt1-cohowelldfn}, these have the
 property that for any  simplicial  $k$-vector space $V$ with finite-dimensional normalisation,
$$
\pi_m(F(k\oplus V)\by_{G(k\oplus V)}\{0\}) \cong \H^{-m}(F/G\ten V),
$$
where $V^2=0$ and 
$$
H^i(F/G\ten V):=\bigoplus_{n \ge 0} \H^{i+n}(F/G) \ten \pi_n(V).
$$

If $G=\bt$ (the one-point set), we write $\H^j(F):= \H^j(F/\bt)$.

We now have the following characterisation of obstruction theory:

\begin{theorem}\label{robs}
If  $\alpha:F \to G$ in $sc\Sp$ is quasi-smooth, then for any small extension $e:I \to A \xra{f} B$ in $s\C_{\L}$, there is a sequence of sets
$$
\pi_0(FA)\xra{f_*} \pi_0(FB\by_{GB}GA) \xra{o_e}  \H^1(F/G\ten I)
$$  
exact in the sense that the fibre of $o_e$ over $0$ is the image of $f_*$. Moreover,  there is a group action of $\H^0(F/G \ten I)$ on $\pi_0(FA)$ whose orbits are precisely the fibres of $f_*$. 

For any $y \in F_0A$, with $x=f_*y$, the fibre of $FA \to FB\by_{GB}GA$ over $x$ is isomorphic to $\ker(\alpha: FI \to GI)$, and the sequence above 
extends to a long exact sequence
$$\xymatrix@R=0ex{
\cdots  \ar[r]^-{f_*}&\pi_n(FB\by_{GB}GA,x) \ar[r]^-{o_e}& \H^{1-n}(F/G \ten I) \ar[r]^-{\pd_e} &\pi_{n-1}(FA,y)\ar[r]^-{f_*}&\cdots\\ 
\cdots \ar[r]^-{f_*}&\pi_1(FB\by_{GB}GA,x) \ar[r]^-{o_e}& \H^0(F/G \ten I)  \ar[r]^-{-*y} &\pi_0(FA).
}
$$
\end{theorem}
\begin{proof}
\cite{ddt1} Theorem \ref{ddt1-robs}.
\end{proof}

\begin{corollary}\label{weak} 
A map $\alpha:F \to G$ of quasi-smooth  $F,G\in sc\Sp$ is a weak equivalence if and only if 
the maps $\H^j(\alpha):\H^j(F) \to\H^j(G)$ are all isomorphisms.
\end{corollary}

\begin{corollary}\label{cohosmoothchar}
If $\alpha:F \to G$ is  quasi-smooth in $sc\Sp$, then $\alpha$ is smooth if and only if $\H^i(F/G)=0$ for all $i>0$.
\end{corollary}

\begin{proposition}\label{longexact}
Let $X,Y, Z:s\C_{\L}\to \bS$ be left-exact functors, with  $X \xra{\alpha} Y$ and $Y \xra{\beta} Z$ quasi-smooth.  There is then a long exact sequence
$$
\ldots \xra{\pd} \H^j(X/Y) \to \H^j(X/Z) \to \H^j(Y/Z) \xra{\pd} \H^{j+1}(X/Y) \to \H^{j+1}(X/Z) \to \ldots
$$
\end{proposition}
\begin{proof}
\cite{ddt1} Proposition \ref{ddt1-longexact}.
\end{proof}

\subsection{Model structures}\label{model}

\begin{theorem}\label{scspmodel}
There is a simplicial model structure on $sc\Sp$, for which the   fibrations are quasi-smooth morphisms, and weak equivalences between quasi-smooth objects are those given in Definition \ref{weakdef}.
\end{theorem}
\begin{proof}
This is \cite{ddt1} Theorem \ref{ddt1-scspmodel}.
\end{proof}

Thus the homotopy category $\Ho(sc\Sp)$ is equivalent to the category of quasi-smooth objects in $sc\Sp$, localised at the weak equivalences of Definition \ref{weakdef}.  

\begin{definition}\label{generalcoho}
Given any morphism $f:X \to Z$, we define $\H^n(X/Z):= \H^n(\hat{X}/Z)$, for $X \xra{i} \hat{X} \xra{p} Z$ a factorisation of $f$ with $i$ a geometric trivial cofibration, and $p$ a geometric fibration. 
\end{definition}

\subsubsection{Homotopy representability}

\begin{definition}\label{schless}
Define the category $\cS$ to consist of functors $F: s\C_{\L}\to \bS$ satisfying the following conditions:
\begin{enumerate}

\item[(A0)] $F(k)$ is contractible.

\item[(A1)] For all small extensions $A \onto B$ in $s\C_{\L}$, and maps $C \to B$ in $s\C_{\L}$, the 
map 
$F(A\by_BC) \to F(A)\by_{F(B)}^hF(C)$ is a  
weak equivalence, where $\by^h$ denotes homotopy fibre product.

\item[(A2)] For all acyclic small extensions $A \onto B$ in $s\C_{\L}$, the map $F(A) \to F(B)$ is a weak equivalence.
\end{enumerate}

Say that a natural transformation $\eta:F \to G$ between such functors is a weak equivalence if the maps $F(A) \to G(A)$ are weak equivalences for all $A\in s\C_{\L}$, and let $\Ho(\cS)$ be the category obtained by formally inverting all weak equivalences in $\cS$.
\end{definition}

\begin{remark}\label{byh}
We may apply the long exact sequence of homotopy  to describe the homotopy groups of homotopy fibre products. 
If $f:X \to Z$, $g:Y \to Z$ in $\bS$ and $P=X\by_Z^hY$,   the map $\theta:\pi_0(P)\to \pi_0(X)\by_{\pi_0(Z)}\pi_0(Y)$ is surjective. Moreover, $\pi_1(Z,*)$ acts transitively on the fibres of $\theta$ over $* \in \pi_0Z$.

Take $v \in \pi_0(P)$ over $*$. Then 
there is
a connecting homomorphism $\pd:\pi_n(Z,*)\to \pi_{n-1}(P,v)$ for all $n\ge 1$, 
giving a long exact sequence
$$
\ldots \xra{\pd} \pi_n(P,v)\to \pi_n(X,v)\by \pi_n(Y,v)\xra{f\cdot g^{-1}} \pi_n(Z,*)\xra{\pd}\pi_{n-1}(P,v)  \ldots .
$$
\end{remark}

\begin{theorem}\label{schrep}
There is a canonical equivalence between the geometric homotopy category $\Ho(sc\Sp)$ and the category $\Ho(\cS)$.
\end{theorem}
\begin{proof}
This is \cite{ddt1} Theorem \ref{ddt1-schrep}.
\end{proof}

\subsubsection{Equivalent formulations}

If $k$ is a field of characteristic $0$, then we may work with dg algebras rather than simplicial algebras.

\begin{definition}
Define $dg\C_{\L}$ to be the category of Artinian local differential $\N_0$-graded  graded-commutative  $\L$-algebras with residue field $k$.
\end{definition}

\begin{definition}
Define a map $A \to B$ in $dg\C_{\L}$ to be a small extension if it is surjective and the kernel $I$ satisfies $I\cdot \m(A)=0$. 
\end{definition}

\begin{definition}\label{sdgspdef}
Define $sDG\Sp$ to be the category of left exact functors from $dg\C_{\L}$ to $\bS$.
\end{definition}

\begin{definition}
Say a map $X \to Y$ in $sDG\Sp$ is quasi-smooth if for all small extensions $f:A \to B$ in $dg\C_{\L}$, the morphism
$$
X(A)\to Y(A)\by_{Y(B)}X(B)
$$ 
is a fibration in $\bS$, which is moreover a  trivial fibration if $f$ is acyclic. 
\end{definition}

\begin{definition}\label{weakdef2}
We will say that a morphism $\alpha: F \to G$ of quasi-smooth objects of $sDG\Sp$ is a weak equivalence if, for all $A \in s\C_{\L}$, the maps $\pi_i F(A) \to \pi_iG(A)$ are isomorphisms for all $i$.
\end{definition}

\begin{proposition}\label{sdgmodel}
There is a model structure on $sDG\Sp$, for which the   fibrations are quasi-smooth morphisms, and weak equivalences between quasi-smooth objects are those given in Definition \ref{weakdef2}.
 \end{proposition}
\begin{proof}
This is \cite{ddt1} Proposition \ref{ddt1-sdgmodel}.
\end{proof}

Most of the constructions from $s\C_{\L}$ carry over to $dg\C_{\L}$. However, there is no straightforward analogue of Definition \ref{underline}.

\begin{definition}\label{qrat}
Define the normalisation functor $N: s\C_{\L} \to dg\C_{\L}$ by mapping $A$ to its associated normalised complex $NA$, equipped with the Eilenberg-Zilber shuffle product (as in \cite{QRat}).
\end{definition}

\begin{definition}
Define $\Spf N^*:  sDG\Sp \to sc\Sp$ by mapping $X: dg\C_{\L} \to \bS$ to the composition $X\circ N:s\C_{\L} \to \bS$. Note that this is well-defined, since $N$ is left exact.
\end{definition}

\begin{theorem}\label{nequiv}
$\Spf N^*:  sDG\Sp \to sc\Sp$ is a right Quillen equivalence. 
\end{theorem}
\begin{proof}
This is \cite{ddt1} Theorem \ref{ddt1-nequiv}.
\end{proof}

In particular, this means that $\Spf N^*$ maps quasi-smooth morphisms to quasi-smooth morphisms, and induces an equivalence $ \oR\Spf N^*: \Ho( sDG\Sp) \to \Ho(sc\Sp)$.

Now assume that $\L=k$.
\begin{theorem}\label{allequiv}
$\Ho(sDG\Sp)$ is equivalent to the category of $L_{\infty}$-algebras localised at tangent quasi-isomorphisms (as considered in \cite{Kon}). This is also equivalent to the category of DGLAs (see \S \ref{dgla}) localised at quasi-isomorphisms.
\end{theorem}
\begin{proof}
Combine \cite{ddt1} Proposition \ref{ddt1-tqiscor} and Corollary \ref{ddt1-mcallequiv}.
\end{proof}

\section{Monads and comonads}\label{algsn}

\subsection{Algebras and coalgebras}

\begin{definition}
A monad (or triple) on a category $\cB$ is a monoid in the category of endofunctors of $\cB$ (with the monoidal structure given by composition of functors). A comonoid (or cotriple) is a comonoid in the category of endofunctors of $\cB$.
\end{definition}

The following is standard:
\begin{lemma}\label{adjnmonad}
Take an adjunction
$$
\xymatrix@1{\cA \ar@<1ex>[r]^G_{\top} & \cE \ar@<1ex>[l]^F}
$$
with unit $\eta:\id \to GF$ and co-unit $\vareps:FG \to \id$. Then $\top:=GF$ is a monad with unit $\eta$ and multiplication $\mu:=G\vareps F$, while $\bot:= FG$ is a comonad, with  co-unit $\vareps$ and comultiplication $\Delta:=F\eta G$.
\end{lemma}

\begin{definition}
Given a monad $(\top,\mu, \eta) $ on a category $\cE$, 
define the category $\cE^{\top}$ of $\top$-algebras to have objects
$$
\top E \xra{\theta} E,
$$
such that $\theta\circ \eta_E=\id$ and $\theta \circ \top \theta= \theta \circ  \mu_E$.
	
A morphism 
$$
g: ( \top E_1 \xra{\theta} E_1 ) \to  (\top E_2 \xra{\phi} E_2)  
$$	
of $\top$-algebras is a morphism $g:E_1 \to E_2$ in $\cE$ such that $\phi\circ \top g= g \circ \theta$.
\end{definition}

We define the comparison functor $K:\cA \to \cE^{\top}$ by
$$
B \mapsto ( UFUB \xra{U\vareps_B}  UB )
$$
on objects, and $K(g)=U(g)$ on morphisms.

\begin{definition}
The adjunction 
$$
\xymatrix@1{\cA \ar@<1ex>[r]^U_{\top} & \cE \ar@<1ex>[l]^F},
$$
is said to be  monadic (or tripleable) if $K:\cD \to \cE^{\top}$ is an equivalence.
\end{definition}

\begin{examples}
Intuitively, monadic adjunctions correspond to algebraic theories, such as the adjunction
$$
\xymatrix@1{\Ring \ar@<1ex>[r]^U_{\top} &\Set  \ar@<1ex>[l]^{\Z[-]},}
$$
 between rings and sets, $U$ being the forgetful functor. Other examples are $k$-algebras   over $k$-vector spaces, or groups over sets.
\end{examples}

\begin{definition}
Dually, given a comonad $(\bot, \Delta, \vareps)$ on a category $\cA$, we  define the category $\cA_{\bot}$ of $\bot$-coalgebras by 
$$
(\cA_{\bot})^{\op} := (\cA^{\op})^{\bot^{\op}},
$$
noting that $\bot^{\op}$ is a monad on the opposite category $\cA^{\op}$. The adjunction of Lemma \ref{adjnmonad} is said to be comonadic (or cotripleable) if the adjunction on opposite categories is monadic.
\end{definition}

\begin{example}
If $X$ is a topological space (or any site with enough points) and $X'$ is the set of points of $X$, let  $u:X'\to X$ be the associated morphism.
Then the adjunction $u^{-1}\dashv u_*$ on sheaves is comonadic, so the category of sheaves  on $X$ is equivalent  $u^{-1}u_*$-coalgebras in  the  category of sheaves  (or equivalently presheaves)  on $X'$

A more prosaic example is that for any ring $A$, the category of $A$-coalgebras is comonadic over the category of $A$-modules.
\end{example}

\subsection{Bialgebras }\label{bialgsn}

As in \cite{osdol} \S IV, take a category $\cB$ equipped with both a monad $(\top, \mu, \eta)$ and a comonad $(\bot, \Delta, \gamma)$, together with a distributivity transformation
$\lambda: \top\bot \abuts \bot \top$
for which the following diagrams commute:
$$
\xymatrix{
\top\bot \ar@{=>}[rr]^{\lambda} \ar@{=>}[d]_{\top\Delta}& &  \bot \top\ar@{=>}[d]_{\Delta\top} \\
\top\bot^2 \ar@{=>}[r]^{\lambda\bot} &\bot\top\bot \ar@{=>}[r]^{\bot\lambda} & \bot^2\top
}
\quad
\xymatrix{
\top\bot \ar@{=>}[rr]^{\lambda} & &  \bot \top \\
\top^2\bot\ar@{=>}[u]_{\mu\bot}  \ar@{=>}[r]^{\top\lambda} & \top\bot\top \ar@{=>}[r]^{\lambda\top}& \bot\top^2\ar@{=>}[u]_{\bot\mu}
}
$$

$$
\xymatrix{
\top\bot \ar@{=>}[rr]^{\lambda} \ar@{=>}[dr]_{\top\gamma}& &  \bot \top\ar@{=>}[dl]^{\gamma\top} \\
& \top} 
\quad
\xymatrix{
\top\bot \ar@{=>}[rr]^{\lambda} & & \bot \top \\
& \bot\ar@{=>}[ul]^{\eta\bot}  \ar@{=>}[ur]_{\bot\eta}.}
$$

\begin{definition}\label{bialgdef}
Given a distributive monad-comonad pair $(\top, \bot)$ on a category $\cB$, define the category $\cB^{\top}_{\bot}$ of bialgebras as follows. The objects of  $\cB^{\top}_{\bot}$ are triples $(\alpha, B, \beta)$ with $(\top B \xra{\alpha} B)$ an object of $\cB^{\top}$ and $B \xra{\beta} \bot B$ an object of $\cB_{\bot}$, such that the composition  $(\beta \circ \alpha):\top B \to \bot B$ agrees with the composition 
$$
\top B \xra{\top \beta} \top\bot B \xra{\lambda} \bot \top B \xra{\bot \alpha} \bot B. 
$$  

A morphism $f:(\alpha, B, \beta) \to (\alpha', B', \beta')$ is a morphism $f: B \to B'$ in $\cB$ such that $\alpha'\circ \top f = f \circ \alpha$ and $\beta' \circ f = \bot f \circ \beta$.  
\end{definition}

To understand how the data $(\top, \bot, \eta, \mu, \gamma, \Delta, \lambda)$ above occur naturally, note that by \cite{osdol} \S IV or \cite{paper2} \S \ref{paper2-gensdc}, these data are equivalent to a diagram
$$
\xymatrix@=8ex{
\cD \ar@<1ex>[r]^{U}_{\top} \ar@<-1ex>[d]_{V} 
&\ar@<1ex>[l]^{F} \cE  \ar@<-1ex>[d]_{V} 
\\
\ar@<-1ex>[u]_{G}^{\dashv}	\cA \ar@<1ex>[r]^{U}_{\top} 
&\ar@<1ex>[l]^{F} \ar@<-1ex>[u]_{G}^{\dashv} \cB,  
}
$$
with $F\dashv U$ monadic, $G\vdash V$ comonadic and $U,V$ commuting with everything (although $G$ and $F$ need not commute). The associated monad is $\top=UF$, and the comonad $\bot= VG$. Distributivity ensures that $\cD \simeq \cE^{\top}\simeq (\cB_{\bot})^{\top}$ and $\cD\simeq \cA_{\bot}\simeq (\cB^{\top})_{\bot}$. In other words, $\cD \simeq \cB^{\top}_{\bot}$.

\begin{example}
If $X$ is a topological space (or any site with enough points) and $X'$ is the set of points of $X$, let $\cD$ be the category of sheaves  of rings on $X$. If $\cB$ is the category of sheaves (or equivalently presheaves) of sets on $X'$, then the description above characterises $\cD$ as a category of bialgebras over $\cB$, with the comonad being $u^{-1}u_*$ for $u:X'\to X$, and the monad  being the free polynomial  functor. 
\end{example}

\section{Constructing SDCs}\label{constructsdc}

Recall the definition of an SDC:
\begin{definition}\label{sdcdef} A simplicial deformation complex $E^{\bullet}$ consists of smooth left-exact functors $E^n:\C_{\L} \to \Set$ for each $n \ge 0$, together with maps 
$$
\begin{matrix}
\pd^i:E^n \to E^{n+1} & 1\le i \le n\\
\sigma^i:E^{n}\to E^{n-1} &0 \le i <n,
\end{matrix}
$$
an associative product $*:E^m \by E^n \to E^{m+n}$, with identity $1: \bullet \to E^0$, where $\bullet$ is the constant functor $\bullet(A)=\bullet$ (the one-point set) on $\C_{\L}$, such that:
\begin{enumerate}
\item $\pd^j\pd^i=\pd^i\pd^{j-1}\quad i<j$.
\item $\sigma^j\sigma^i=\sigma^i\sigma^{j+1} \quad i \le j$.
\item 
$
\sigma^j\pd^i=\left\{\begin{matrix}
			\pd^i\sigma^{j-1} & i<j \\
			\id		& i=j,\,i=j+1 \\
			\pd^{i-1}\sigma^j & i >j+1
			\end{matrix} \right. .
$
\item $\pd^i(e)*f=\pd^i(e*f)$.
\item $e*\pd^i(f)=\pd^{i+m}(e*f)$, for $e \in E^m$.
\item $\sigma^i(e)*f=\sigma^i(e*f)$.
\item $e*\sigma^i(f)=\sigma^{i+m}(e*f)$, for $e \in E^m$.
\end{enumerate}
\end{definition}

From the viewpoint of homotopical algebra, there is a more natural way of characterising the smoothness criterion  for $E^{\bullet}$. Analogously to \cite{sht} Lemma VII.4.9, we define matching objects by $M^{-1}E:=\bullet$, $M^0E:=E^0$, and for $n>0$
$$
M^nE=\{(e_0,e_1,\ldots,e_n) \in (E^n)^{n+1}\,|\, \sigma^ie_j=\sigma^{j-1}e_i\, \forall i<j \}.
$$

\begin{proposition}\label{sdcmatch}
The  canonical maps $\underline{\sigma}:E^{n+1} \to M^nE$, given by $e \mapsto (\sigma^0e, \sigma^1e, \ldots , \sigma^{n}e)$, are all  smooth, for $n\ge 0$.
\end{proposition}
\begin{proof}
 Since $E^{n}$ is smooth, by the Standard Smoothness Criterion (e.g. \cite{Man} Proposition 2.17)  it suffices to show that this is surjective on tangent spaces. The tangent space of $M^{n}E$ consists of $(n+1)$-tuples $\gamma_i \in C^n(E)$ satisfying $\sigma^i\gamma_j =\sigma^{j-1}\gamma_i$, for $i <j$. For any cosimplicial complex $C^{\bullet}$, there is a decomposition of the associated cochain complex as $C^n=N_c^n(C) \oplus D^n(C)$, where $N_c^n(C)=\cap_{i=0}^{n-1}\ker \sigma^i$, and $D^n(C)=\sum_{i=1}^n\pd^i C^{n-1}$. Moreover $\underline{\sigma}:D^{n} \to M^{n-1}C$ is an isomorphism, giving the required surjectivity.  
\end{proof}

\begin{definition}\label{sdcmcdef}
Given an SDC $E$, recall from \cite{paper2} that  the Maurer-Cartan functor $\mc_E: \C_{\L} \to \Set$ is defined by 
$$
\mc_E(A)=\{\omega \in E^1(A) \colon \omega*\omega=\pd^1(\omega)\}.
$$
The group $E^0(A)$ acts on this by conjugation, and we define $\Def_E(A)$ to be the groupoid with objects $\mc_E(A)$ and morphisms given by $E^0(A)$ via this action. We say that an SDC \emph{governs} a deformation problem if $\Def_E$ is equivalent to the associated deformation functor.
\end{definition}

\begin{definition}\label{sdccoho} Recall that  $\CC^{\bullet}(E)$ denotes the tangent space of $E^{\bullet}$, i.e. $\CC^n(E)=E^n(k[\eps])$ for $\eps^2=0$. This has the natural structure of a cosimplicial complex, by \cite{paper2}, and we set $\H^i(E):=\H^i(\CC^{\bt}(E))$.
\end{definition}

\subsection{SDCs from bialgebraic structures}

\begin{definition}\label{delta**}
Recall from \cite{monad} that  $\Delta_{**}$ is defined to be the subcategory of the ordinal number category $\Delta$ containing only  those non-decreasing morphisms $f:\mathbf{m} \to \mathbf{n}$ with $f(0)=0, f(m)=n$. We define a monoidal structure on this category by setting $\mathbf{m}\ten \mathbf{n}= \mathbf{m+n}$, with 
$$
(f\ten g)(i)= \left\{ \begin{matrix} f(i) & i\le m \\ g(i-m)+n & i \ge m, \end{matrix} \right.
$$ 
for $f:\mathbf{m} \to \mathbf{n}$. 
\end{definition}

\begin{definition}
As in \cite{monad}, define   monoidal structures on $\Set^{\Delta_{**}}$ and $\bS^{\Delta_{**}}$ by setting
$$
(X\ten Y)^n := \coprod_{a+b=n} X^a\ten Y^b,
$$
with operations given by 
\begin{eqnarray*}
\pd^i(x \ten y) &=& \left\{\begin{matrix}(\pd^ix)\ten y & i\le a \\ x\ten (\pd^{i-a}y) & i>a; \end{matrix} \right.\\
\sigma^i(x\ten y)&=& \left\{\begin{matrix}(\sigma^ix)\ten y & i< a \\ x\ten (\pd^{i-a}y) & i\ge a. \end{matrix} \right.
\end{eqnarray*}
The identity $I$ is given by $I^0=1$ and $I^n=\emptyset$ for $n>0$.
\end{definition}

Note that an SDC over $\L$ is a smooth left-exact functor from $\C_{\L}$ to the category of monoids in $\Set^{\Delta_{**}}$.

Assume that we have a diagram 
$$
\xymatrix@=8ex{
\cD \ar@<1ex>[r]^{U}_{\top} \ar@<-1ex>[d]_{V} 
&\ar@<1ex>[l]^{F} \cE  \ar@<-1ex>[d]_{V} 
\\
\ar@<-1ex>[u]_{G}^{\dashv}	\cA \ar@<1ex>[r]^{U}_{\top} 
&\ar@<1ex>[l]^{F} \ar@<-1ex>[u]_{G}^{\dashv} \cB,  
}
$$
of homogeneous (i.e. preserving fibre products, but not the final object) functors from $\C_{\L}$ to $\Cat$
as in \cite{paper2} \S \ref{paper2-gensdc} (i.e. $\cB$ has uniformly trivial deformation theory, with the diagram satisfying the conditions of \S \ref{bialgsn}).   Recall that we write
$
\toph=UF$, 	$\both=FU$,
$\botv=VG$, and	$\topv=GV$.

\begin{proposition}\label{enrichtopbot}
For the diagram above and $A \in \C_{\L}$, $\cB(A)$ has the structure of category enriched in $\Set^{\Delta_{**}}$, with 
$$
\cHom_{\cB}(B,B')^n= \Hom_{\cB}(\toph^n B, \botv^n B').
$$
\end{proposition}
\begin{proof}
\cite{monad} Proposition \ref{monad-enrichtopbot}.
\end{proof}

\begin{examples}\label{keyegs}
\begin{enumerate}
\item
If $X$ is a topological space (or any site with enough points) and $X'$ is the set of points of $X$, let $\cD(A)$ be the category of sheaves  of flat $A$-algebras on $X'$. If $\cB$ is the category of sheaves (or equivalently presheaves) of flat $A$-modules on $X'$, then the description above characterises $\cD$ as  $\cB^{\toph}_{\botv}$, with  $\botv=u^{-1}u_*$ for $u:X'\to X$, and  $\toph$ being the free $A$-algebra functor for module. This example arises when considering deformations of a scheme $X$ in \cite{paper2} \S 3.2, since deformations of $X$ are equivalent to deformations of the sheaf $\O_X$ of algebras.

\item
Another important example (considered in \cite{paper2} \S 3.1) is when $\cD(A)$ is the category of flat Hopf algebras over $A$,  with $\cA(A)$, $\cE(A)$  and $\cB(A)$ the categories of flat algebras, coalgebras and modules, respectively.  

\item\label{keyeg} 
In order to make the first example functorial, we could let $\cB$ be the category of pairs $( \{M_x\}_{x \in X},X)$, for $X$ a topological space and $\{M_x\}_{x \in X} $ a presheaf of flat $A$-modules on $X'$, with a morphism $f^{\sharp}:  ( \{N_y\}_{y \in Y},Y)\to ( \{M_x\}_{x \in X},X)$  given by a map $f:X \to Y$ of topological spaces, together with maps $f^{\sharp}_x: N_{f(x)} \to M_x$ for all $x \in X$.

We may define $\botv$ and $\toph$ as before, and then $\cB^{\toph}_{\botv}$ will be the category of pairs $( \O_X,X)$, where $X$ is a topological space and $\O_X$ a sheaf of flat $A$-algebras on $X$.
\end{enumerate}
\end{examples}

\subsection{SDCs from diagrams}\label{diagram}

\begin{definition}
Given a morphism $f:D \to D'$ in $\cD(k)$, choose  lifts   $B,B' \in \cB(\L)$ of $UVD, UVD'\in \cB(k)$ (which exist since the deformation theory of $\cB$ is uniformly trivial). Then  define
$$
E^n_{\cD/\cB}(f):=\underline{\hom}^n(B,B')_{UV(\alpha^n_{D'}\circ f \circ \vareps^n_D)}: \C_{\L} \to \Set,
$$
where $\alpha:1 \to \topv $ and $\vareps: \both \to 1$ are the unit and counit of the respective adjunctions.

Write $E^*_{\cD/\cB}(D):= E^*_{\cD/\cB}(\id_D)$. 
Note that uniform triviality of $\cB$ ensures that these constructions are independent of the choices of lift, since any other choice is isomorphic.

\end{definition}

\begin{lemma}\label{cosimplicialE}
$E_{\cD/\cB}(f)(A)$ has the natural structure of a cosimplicial complex.

For every pair of composable morphisms $f,g$ in $\cD(k)$ between such objects, there is a product 
$$
E_{\cD/\cB}(f)(A)\ten E_{\cD/\cB}(g)(A) \to E_{\cD/\cB}(f\circ g)(A)
$$
in $\Set^{\Delta_{**}} $, functorial in $A$. 
\end{lemma}
\begin{proof}
It follows from Proposition \ref{enrichtopbot} that $E_{\cD/\cB}(f)(A) \in Set^{\Delta_{**}} $, with operations
\begin{eqnarray*}
\pd^i(x) &=&    \botv^{i-1}V\alpha_{G\botv^{n-i}B}\circ	x \circ \toph^{i-1}U\vareps_{F\toph^{n-i}B},\quad 0< i \le n \\
\sigma^i(x) &=&	 \botv^{i}\gamma_{\botv^{n-i-1}B}\circ	x \circ \toph^{i}\eta_{\toph^{n-i-1}B}, \quad 0 \le i < n,
\end{eqnarray*}
for $\eta: 1 \to \toph$ and $\gamma: \botv \to 1$ the respective unit and co-unit. The multiplication also follows from Proposition \ref{enrichtopbot}.

The canonical object of $\mc(E_{\cD/\cB}(\id_D)(\L))$ corresponding to $D$ gives an element $\omega_D \in E_{\cD/\cB}(\id_D)(\L)^1$ and 
we then enhance the structure above to give a cosimplicial structure by setting
$$
\pd^0x:= \omega_{D'} * x \quad \pd^{n+1}x= x* \omega_D,
$$
for $x \in  E_{\cD/\cB}(f)(A)^n$.
\end{proof}

\begin{definition}\label{powersdc} 
Given an SDC $E$, and a simplicial set $X$, define an SDC $E^X$ by 
$$
(E^X)^n=(E^n)^{X_n}.
$$
For $x \in X_{n+1}$, $y \in Y_{n+1}$, $z \in X_{m+n}$, $1 \le i \le n$, $0 \le j < n$, $e \in (E^X)^n$ and $f \in (E^X)^m$, we define the operations by
\begin{eqnarray*}
\pd^i(e)(x)&:=& \pd^i(e(\pd_i x))\\
\sigma^j(e)(y)&:=& \sigma^j(e(\sigma_i y)),\\ 
(f*e)(z)&:=& f((\pd_{m+1})^nz)*e((\pd_{0})^mz).
\end{eqnarray*}
\end{definition}

\begin{definition}
Let $\CC^{\bt}_{\cD/\cB}(f)$ be the 
 the tangent space of $E_{\cD/\cB}(f)$. This is a vector space over $k$, and we 
 define 
$$
\Ext^*_{\cD/\cB}(f):= \H^*(\CC^{\bt}_{\cD/\cB}(f)); 
$$
this construction is closely related to Van Osdol's bicohomology (\cite{osdol}). 
\end{definition}

\begin{lemma}\label{powercoho}
If $X$ is a finite simplicial set, then
$$
\H^n(E^X)\cong \bigoplus_{i+j=n} \H^i(E)\ten \H^j(X,k).
$$
\end{lemma}
\begin{proof}
Since $X$ is finite, $\CC^{\bt}(E^X)\cong\CC^{\bt}(E)\ten k^X$, and the result now follows from the K\"unneth formula.
\end{proof}

\begin{definition}\label{sdcdiagram}
Given a small category $\bI$ and an $\bI$-diagram $\bD:\bI \to \cD(k)$, 
define the SDC $E^{\bt}_{\cD/\cB}(\bD)$ by
$$
E^n_{\cD/\cB}(\bD)= \prod_{\substack{i_0 \xra{f_1} i_1 \xra{f_2} \ldots \xra{f_n} i_n\\ \text{in }\bI}} E^n(\bD(f_n\circ f_{n-1}\circ \ldots f_0)) = \prod_{x \in B\bI_n} E^n( \bD(\pd_1^{\phantom{1}n-1}x)),
$$
where $B\bI$ is the nerve of $\bI$ (so $B\bI_0=\Ob(\bI)$,  $B\bI_1=\Mor(\bI)$), and  $\pd_1^{\phantom{1}-1}:=\sigma_0$.

We define the operations by the formulae of Definition \ref{powersdc}.
\end{definition}

\begin{theorem}\label{diagramdeform}
The SDC $E^{\bt}_{\cD/\cB}(\bD)$ governs deformations of the diagram $\bD:\bI \to \cD(k)$.
\end{theorem}
\begin{proof}
This follows immediately from \cite{monad} Lemma \ref{monad-sdcdiagramsub}, which characterises objects of $\Def_E$ as diagrams from $\bI$ to a category equivalent to $\cD(A)$.
\end{proof}

\begin{lemma}\label{extcohocalc}
Given a diagram $\bD:\bI \to \cD(k)$, the cohomology groups $\H^*(E^{\bt}_{\cD/\cB}(\bD))$ are given by hypercohomology of the bicomplex
$$
\prod_{i \in \Ob \bI} \CC^{\bt}_{\cD/\cB}(\bD(i)) \xra{f_*-f^*} 
\ldots \to \prod_{\substack{i_0 \xra{f_1} i_1 \xra{f_2} \ldots \xra{f_n} i_n\\ \text{in }N_n\bI}} \CC^{\bt}_{\cD/\cB}(\bD(f_n\circ f_{n-1}\circ \ldots f_0))) \to \ldots,
$$ 
where $N_n\bI \subset B_n\bI$ consists of non-degenerate simplices, or equivalently strings of non-identity morphisms. 
\end{lemma}
\begin{proof}
Since $\CC^{\bt}(E^{\bt}_{\cD/\cB}(\bD))$ is the diagonal of the bicosimplicial complex
$$
[m,n]\mapsto  \prod_{\substack{i_0 \xra{f_1} i_1 \xra{f_2} \ldots \xra{f_n} i_n\\ \text{in }\bI}}\CC^{m}_{\cD/\cB}(\bD(f_n\circ f_{n-1}\circ \ldots f_0)), 
$$
the Eilenberg-Zilber theorem implies that it is homotopy equivalent to the total complex of the associated binormalised complex. The vertical normalisation is just given by replacing $B_n\bI$ with $N_n \bI$. 
\end{proof}

\begin{example}\label{extscheme}
If we define $\cD$ and $\cB$ as in Example \ref{keyegs}.\ref{keyeg}, then the category of flat schemes over $A$ is a full subcategory of $\cD(A)$, closed under deformations. Therefore Theorem \ref{diagramdeform} constructs an SDC governing deformations of a diagram of schemes.

For a morphism  $f:(X, \O_X) \to (Y, \O_Y)$ in $\cD(k)^{\op}$, the reasoning of \cite{paper2} \S \ref{paper2-scheme} adapts to show that
$$
\Ext^*_{\cD/\cB}(f)= \EExt^*_{\O_Y}( \oL_{\bt}^{Y/k}, \oR f_* \O_X)= \EExt^*_{\O_X} (f^*\oL_{\bt}^{Y/k},\O_X),
$$
where $\oL_{\bt}^{Y/k}$ is the cotangent complex of \cite{Ill1}.
\end{example}

\subsection{Constrained deformations}\label{constrain}

We now consider a generalisation of \S \ref{diagram}, by taking a small diagram 
$$
\bD: \bI\to \cD(k), 
$$ 
 a subcategory  $\bJ \subset \bI$, and $\widetilde{\bD|_{\bJ}}:\bJ \to \cD(\L)$ lifting $\bD|_{\bJ}$. We wish to describe deformations of $\bD$ which agree with $\widetilde{\bD|_{\bJ}}$ on $\bJ$. Note that when $\bI=(0 \to 1)$ and $\bJ=\{1\}$, this is the type of problem considered in  \cite{cones} and \cite{ranlie}.

\begin{proposition}\label{diagramdefs}
Given a $\bI$-diagram $\bD:\bI \to \cD(k)$,  with $\widetilde{\bD|_{\bJ}}$ as above,  the groupoid of deformations of $\bD$ fixing $\widetilde{\bD|_{\bJ}}$ is governed by the SDC
$$
E^{\bt}_{\cD/\cB}(\bD)\by_{E^{\bt}_{\cD/\cB}(\bD|_{\bJ} )}\bt,
$$
where $
\bt \to E^{\bt}_{\cD/\cB}(\bD|_{\bJ} )
$
is defined by the object of $\mc(E^{\bt}_{\cD/\cB}(\bD|_{\bJ} )) $ corresponding to $\widetilde{\bD|_{\bJ}} $. 
\end{proposition}
\begin{proof}
By Theorem \ref{diagramdeform}, it suffices to show that 
$\Def(E^{\bt}_{\cD/\cB}(\bD)\by_{E^{\bt}_{\cD/\cB}(\bD|_{\bJ} )}\bt) $ is equivalent to the $2$-fibre product
$$
\Def(E^{\bt}_{\cD/\cB}(\bD)) \by_{\Def(E^{\bt}(\bD|_{\bJ}))}^h (\{\widetilde{\bD|_{\bJ}} \}, \id).
$$

We know that the functor $\mc$ preserves inverse limits, so 
$$
\mc(E^{\bt}(\bD)\by_{E^{\bt}(\bD|_{\bJ} )}\{\widetilde{\bD|_{\bJ}} \} )= \mc(E^{\bt}_{\cD/\cB}(\bD)) \by_{\mc(E^{\bt}(\bD|_{\bJ}))}  \{\widetilde{\bD|_{\bJ}} \}.
$$

Since $E^0(\bD)(A) \to E^{0}(\bD|_{\bJ} )(A)$ is also surjective (by smoothness), we see that 
$$
\Def(E^{\bt}(\bD)\by_{E^{\bt}(\bD|_{\bJ} )}\{\widetilde{\bD|_{\bJ}} \} )\simeq \Def(E^{\bt}_{\cD/\cB}(\bD)) \by_{\Def(E^{\bt}(\bD|_{\bJ}))}^h  (\{\widetilde{\bD|_{\bJ}} \}, \id).
$$
\end{proof}

\begin{example}\label{overscheme}
Given a morphism $f:X \to Y$ of schemes over $k$, and a flat formal deformation $\fY$ of $Y$ over $\L$, 
we may consider deformations  of $X$ over $\fY$, or equivalently deformations of this diagram fixing $\fY$. Define the diagram $\bD$ to be $f^{\sharp}:( \O_Y, Y)\to (\O_X, X)$ in the category $\cD(k)$  of  example \ref{keyegs}.\ref{keyeg}, and let $\widetilde{\bD|_{\bJ}}$ be the object $(\O_{\fY}, \fY)$ of $\cD(\L)$.  Proposition \ref{diagramdefs} then gives an SDC
$$
E:=E^{\bt}_{\cD/\cB}(\bD)\by_{E^{\bt}_{\cD/\cB}(\bD|_{\bJ} )}\bt
$$
governing this problem.

Lemma \ref{extcohocalc} implies that the tangent complex $\CC^{\bt}(E)$ is the mapping cone of $\CC^{\bt}_{\cD/\cB}(\O_X, X) \to \CC^{\bt}_{\cD/\cB}(f^{\sharp})$, so 
by Example \ref{extscheme}, the cohomology of this SDC is given by
$$
\H^*(E) \cong \EExt^*_{\O_X}(\cone( f^*\oL_{\bt}^{Y/k} \to \oL_{\bt}^{X/k}), \O_X)   \cong  \EExt^*_{\O_X}( \oL_{\bt}^{X/Y},\O_X). 
$$
\end{example}

\begin{example}
We could go further, and let $\bE$ be a diagram  $Z \xra{g} X \xra{f} Y$ over $k$, with a fixed formal deformation $\widetilde{gf}:\fZ \to \fY$ of $fg$ over $\L$. Governing this deformation problem, we get another SDC 
$$
F:= E^{\bt}_{\cD/\cB}(Z \xra{g} X \xra{f} Y)\by_{E^{\bt}_{\cD/\cB}(Z \xra{gf} Y)}\{ \fZ \xra{\widetilde{gf}} \fY\}.
$$

Now, $\CC^{\bt}(F)= \ker(\CC^{\bt}(E_{\cD/\cB}(\bE)) \to \CC^{\bt}(E_{\cD/\cB}(\bE|_{\bJ})))$, and 
Lemma \ref{extcohocalc} implies that 
$\CC^{\bt}(E_{\cD/\cB}(\bE))$ is  homotopy equivalent to the total complex of
$$
\CC^{\bt}_{\cD/\cB}(Z) \by \CC^{\bt}_{\cD/\cB}(X) \by \CC^{\bt}_{\cD/\cB}(Y) \to \CC^{\bt}_{\cD/\cB}(g) \by \CC^{\bt}_{\cD/\cB}(fg) \by \CC^{\bt}_{\cD/\cB}(f) \to \CC^{\bt}_{\cD/\cB}(fg),
$$
while $\CC^{\bt}(E_{\cD/\cB}(\bE|_{\bJ}))$ is  homotopy equivalent to the total complex of
$$
\CC^{\bt}_{\cD/\cB}(Z)  \by \CC^{\bt}_{\cD/\cB}(Y) \to   \CC^{\bt}_{\cD/\cB}(fg),
$$
so $\CC^{\bt}(F)$ is homotopy equivalent to the total complex of
$$
 \CC^{\bt}_{\cD/\cB}(X)  \to \CC^{\bt}_{\cD/\cB}(g)  \by \CC^{\bt}_{\cD/\cB}(f)\to \CC^{\bt}_{\cD/\cB}(fg).
$$

By Example \ref{extscheme}, 
\begin{eqnarray*}
\H^*\ker(\CC^{\bt}_{\cD/\cB}(X)\to \CC^{\bt}_{\cD/\cB}(f)) &\cong&\EExt^*_{\O_X}( \oL_{\bt}^{X/Y}, \O_X)\\
\H^*\ker(\CC^{\bt}_{\cD/\cB}(g)\to \CC^{\bt}_{\cD/\cB}(fg)) &\cong&\EExt^*_{\O_X}( \oL_{\bt}^{X/Y}, \oR g_*\O_Z),
\end{eqnarray*}
and these isomorphisms combine to give
$$
H^*(F) \cong  \EExt^*_{\O_X}( \oL_{\bt}^{X/Y},\cone (\O_X \to \oR g_*\O_Z)[-1]).
$$
Note that this more accurately captures the higher structure than the SDC of \cite{paper2} \S 3.3, whose cohomology had $g_*\O_Z$ in place of $\oR g_*\O_Z$ above.
\end{example}

\section{Extended deformation functors from SDCs}\label{extsdc}

Given an SDC $E$, the aim of this section is to extend the classical deformation groupoid $\Def_E:\C_{\L} \to \Grpd$ of \cite{paper2}  from $\C_{\L}$ to the whole of $s\C_{\L}$. Groupoids turn out to be too restrictive for our purposes, so we will define a simplicial set-valued functor functor $\ddef_E:s\C_{\L} \to \bS$ extending the classifying space $B\Def_E$ of the deformation groupoid. 

For a monad $\top$, the obvious extension of the functor describing deformations of a $\top$-algebra is the functor of deformations of a strong homotopy $\top$-algebra. Strong homotopy algebras were defined by Lada in \cite{loop} to characterise the structures arising on   deformation retracts of  $\top$-algebras in topological spaces, but the description works over any simplicial category. This motivates the following definition:

\begin{definition}\label{mcdef}
Given an SDC $E$, define the  Maurer-Cartan functor  $\mc_E:s\C_{\L}\to \Set$ by
$$
\mc_E(A)\subset \prod_{n\ge 0} E^{n+1}(A^{I^n}),
$$
consisting of those $\underline{\omega}$ satisfying:
\begin{eqnarray*}
\omega_m(s_1,\ldots, s_m)*\omega_n(t_1,\ldots, t_n)&=&\omega_{m+n+1}(s_1,\ldots, s_m,0,t_1,\ldots, t_n);\\ 
\pd^i\omega_n(t_1,\ldots,t_n)&=&\omega_{n+1}(t_1, \ldots,t_{i-1},1,t_i,\ldots,t_n);\\ 
\sigma^i\omega_n(t_1,\ldots,t_n)&=&\omega_{n-1}(t_1, \ldots,t_{i-1},\min\{t_i,t_{i+1}\},t_{i+2},\ldots, t_n);\\
\sigma^0\omega_n(t_1,\ldots,t_n)&=&\omega_{n-1}(t_2, \ldots,t_n);\\
\sigma^{n-1}\omega_n(t_1,\ldots,t_n)&=&\omega_{n-1}(t_1, \ldots,t_{n-1}),\\ 
\sigma^0\omega_0&=&1,
\end{eqnarray*}
where $I:=\Delta^1$. 
\end{definition}

\begin{remarks}\label{sdcexpln}
\begin{enumerate}
\item
One way to think of this construction is that, if we start with an element $\omega \in E^1$ such that $\sigma^0 \omega=1$, then there are $2^n$ elements generated by  $\omega$ in each $E^{n+1}$. To see  this correspondence,  take a vector in $\{0,1\}^n$, then  substitute ``$\omega*$'' for each $0$, and ``$\pd^1$'' for each $1$,  adding a final $\omega$. These elements will be at the vertices of an $n$-cube, and $\omega_n$ is then a homotopy between them.

\item  Lada's definition of a strong homotopy algebra differs slightly in that it omits all of the degeneracy conditions except $\sigma^0\omega_0=0$. Our choices are made so that we work with normalised, rather than unnormalised, cochain complexes associated to a cosimplicial complex. Since these are homotopy equivalent, both constructions will yield weakly equivalent deformation functors, even if we remove all degeneracy conditions. 

\item In \cite{monad} Proposition \ref{monad-rhommc} it is shown that $\mmc$ has a precise homotopy-theoretical interpretation as the derived functor associated to the functor sending an SDC $E$ and $A \in \C_{\L}$ to the set $\mc_E(A)$ from Definition \ref{sdcmcdef}. In the scenario of \S\ref{constructsdc},  it follows from the results of \cite{monad} that for $A \in s\C_{\L}$,  $\mc_E(A)$ is the set of objects of the Segal space of strong homotopy bialgebras over the object being deformed.  
\end{enumerate}
\end{remarks}

\begin{proposition}\label{mcq}
$\mc_E:s\C_{\L} \to \Set$  is quasi-smooth. Moreover, if $f:E\to F$ is a morphism of SDCs  such that the maps $f^n:E^n \to F^n$ are smooth for all $n$, then 
$\mc_E\to \mc_F$  is quasi-smooth.
\end{proposition}
\begin{proof}
This follows immediately from \cite{monad} Lemma \ref{monad-mcqsub}.
\end{proof}

\begin{definition}\label{ddefdef} 
By \cite{paper1} Lemma 1.5, $E^0$ is a group, which we denote by $G_E$.
 Observe that $G_E$ acts on $\mc_E$ by $(g,\omega)\mapsto g*\omega*g^{-1}$. We now define the deformation functor $\ddef_E:s\C_{\L}\to \bS$ by $\ddef_E:=[\underline{\mc}_E/\underline{G}_E]$, for $\uline{X}$ as in Definition \ref{underline}, and $[-,-]$  the homotopy quotient of \S \ref{quotsp}.
\end{definition}

\begin{proposition}
If $A \in \C_{\L}$, then $\ddef_E(A)$ is just the classifying space $B\Def_E(A) \in \bS$ of the deformation groupoid $\Def_E(A)$ from Definition \ref{sdcmcdef}.
\end{proposition}
\begin{proof}
Take $\uline{\omega} \in \mc_E(A)$. Since  $A \in \C_{\L}$,   $A^K=A$ for all connected simplicial sets $K$, so $E^{n+1}(A^{I^n})= E^{n+1}(A)$, and $\omega_n =\omega_0^{*(n+1)} $, with the Maurer-Cartan relations reducing to
$$
\pd^1\omega_0=\omega_0*\omega_0 \quad \sigma^0\omega_0=1.
$$
These are precisely the conditions defining the Maurer-Cartan space $\mc_E(A)$ of Definition \ref{sdcmcdef} , and $\Def_E(A)$ is the groupoid given by the action of $E^0(A)$ on $\mc_E(A)$, as required. 
\end{proof}

\begin{proposition}\label{mcs}
The functor $\ddef_E$ is quasi-smooth. More generally, if $f:E\to F$ is a morphism of SDCs, such that $f^n:E^n \to F^n$ is smooth for all $n$, then
 $\ddef_E\to \ddef_F$ is quasi-smooth.
\end{proposition}
\begin{proof}
This follows immediately from Corollary \ref{fibquot}.
\end{proof}

\begin{proposition}\label{mch}
The cohomology groups $\H^j(\ddef_E)$ are isomorphic to the groups $\H^{j+1}(E)$ from Definition \ref{sdccoho}.
\end{proposition}
\begin{proof}
This follows immediately from \cite{monad} Corollary \ref{monad-cot3}.
\end{proof}

\subsection{Deformations of  morphisms}\label{morphisms}

The problem which we now wish to consider is that of deforming a morphism with fixed endpoints. Assume that we have a category-valued functor $\cD:\C_{\L} \to \Cat$. Fix objects $D,D'$ in  $\cD(\L)$,   
and a morphism $f$ in $\cD(k)$ from $D$ to $D'$. The deformation problem which we wish to consider is to describe, for each $A \in \C_{\L}$, the set of morphisms $f_A:D \to D'$ in $\cD(A)$ deforming $f$.
This amounts to taking the special case $\bI=(0 \to 1) $ and $\bJ=\{0,1\}$ in \S \ref{constrain}.

Now assume that we have a diagram of functors from $\C_{\L}$ to $\Cat$ as in \S \ref{constructsdc}, and consider the cosimplicial complex $F^{\bt}$ in $\Sp$ given by
$
F^{\bt}:= E_{\cD/\cB}^{\bt}(f)$ from  Lemma \ref{cosimplicialE}.

On $s\C_{\L}$, we now define a deformation functor 
$$
\ddef_F(A)\subset \prod_{n\ge 0} \underline{F}^{n}(A)^{\Delta^n},
$$
associated to $F$, to consist of those $\underline{\theta}$ satisfying:
\begin{eqnarray*}
\pd^i\theta_n &=& \eps_{n+1-i}^*\theta_{n+1}\\
\sigma^i\theta_n &=& \eta_{n-1-i}^*\theta_{n-1},
\end{eqnarray*}
for face maps $\eps_i:\Delta^n \to \Delta^{n+1}$ and degeneracy maps $\eta_i:\Delta^n \to \Delta^{n-1}$ defined as in \cite{W} Ch.8. N

\begin{proposition}
$\ddef_F$ is quasi-smooth, and $\H^i(\ddef_F)\cong \H^i(F)$.
\end{proposition}
\begin{proof}
The first statement follows from  \cite{sht} \S VII.5, which shows that the total space functor $\Tot$ from cosimplicial simplicial sets to simplicial sets is right Quillen. The description of cohomology is straightforward.
\end{proof}

\begin{proposition}\label{loopmor}
If $\bI$ is the category  $(0 \xra{m} 1)$,  let $\bD:\bI \to \cD(k)$ be the functor given by $\bD(0)=D, \bD(1)=D'$ and $\bD(m)=f$, then there is a canonical weak equivalence
$$
\ddef_F \simeq \ddef( E_{\cD/\cB}(\bD)\by_{ E_{\cD/\cB}(D) \by E_{\cD/\cB}(D')}\bt),
$$
where $\bt \to E_{\cD/\cB}(D) \by E_{\cD/\cB}(D')$ 
is defined by the object  $(D,D')\in\mc(E^{\bt}(D) \by E^{\bt}(D'))(\L) $.

Thus $\ddef_F$ governs deformations of $f$ which fix $D,D'$. 
\end{proposition}
\begin{proof}
Let $C:= E_{\cD/\cB}(\bD)\by_{ E_{\cD/\cB}(D) \by E_{\cD/\cB}(D')}\bt$.
By \cite{monad} Lemma \ref{monad-loopmorsub}, there are  canonical equivalences $\mmc(C )(A)\simeq \ddef(F)(A)$, so we need only observe that $C^0=1$, so $\ddef(C) = \mmc(C)$. 
The final statement then follows from Proposition \ref{diagramdefs}.
\end{proof}

\subsubsection{Deforming identity morphisms}

If we now consider deformations of the morphism $\id_D:D \to D$, write $F$ for the cosimplicial complex $E_{\cD/\cB}^{\bt}(\id_D)$ governing deformations of $\id_D$, and $E$ for the SDC describing deformations of $D$, as defined in \cite{paper2} \S \ref{paper2-gensdc} (or just by taking the special case $\bI=\bt$ of Definition \ref{sdcdiagram}). Note that $E^n=F^n$, with the operations agreeing whenever they are defined on both. If we write $e:= \pd^01 \in F^1$, note that we also have $\pd^0f= e*f$ and $\pd^{n+1}f=f*e$ for $f \in F^n$.

This gives us an isomorphism $\CC^{\bt}(E)\cong \CC^{\bt}(F)$, and hence $\H^n(\ddef_E)=\H^{n+1}(E)\cong \H^{n+1}(\ddef_F)$. 

\begin{proposition}\label{loopmor2}  
Under the scenario above, the simplicial set  $\ddef_F(A)$ is weakly equivalent to the loop space $\Omega \ddef_E(A)$ of $\ddef_E(A)$ over the point $\omega_D \in \ddef_D(\L)$. This equivalence is  functorial in $A \in s\C_{\L}$
\end{proposition}
\begin{proof}
Define the SDC $PE$  to be the fibre of $\ev_0:E^I \to E$ over the constants $\{e^n\}$. It follows from Lemma \ref{powercoho} that the cohomology groups of $PE$ are all $0$. Now define  the SDC $\Omega E$ to be the fibre of $\ev_1:PE \to E$ over $\{e^n\}$.

By Proposition \ref{loopmor}, $\ddef_F$ is weakly equivalent to  $\ddef (\Omega E)$. By Proposition \ref{mcs}, $\ddef_{PE} \to \ddef_{E}$ is quasi-smooth, and the fibre is $\ddef_{\Omega E}$. Since $\ddef_{PE}$ is contractible, this means that $\ddef_{\Omega E}$ is homotopic to the loop space of $\ddef_E$.
\end{proof}

\begin{remark}
Note that we can describe $\Omega E$ entirely in terms of the structure on $F$, since 
$$
(\Omega E)^n =  (F^n)^n,
$$
with 
\begin{eqnarray*}
\pd^i(f_1, \ldots, f_n)&=& (\pd^if_1,\pd^if_2, \ldots,\pd^if_i, \pd^i f_i, \ldots,\pd^i  f_n)\\
\sigma^i(f_1, \ldots, f_n)&=&(\sigma^if_1,\sigma^if_2, \ldots,\sigma^if_i, \sigma^if_{i+2}, \ldots, \sigma^i f_n)\\
(g_1, \ldots, g_m)*(f_1, \ldots f_n)&=& (g_1*e^n, \ldots, g_m*e^n,e^m*f_1, \ldots e^m*f_n ). 
\end{eqnarray*}

Now, given any smooth object $F \in c\Sp$,  we may regard $F$ as a cosimplicial complex of smooth objects in $\Sp$ (as in  \cite{ddt1} Definition \ref{ddt1-spfdef}), and then 
$$
\underline{F} = \ddef(F) \simeq \ddef(\Omega E).
$$

This means that we cannot expect derived deformation functors coming from SDCs to have any more structure than arbitrary deformation functors.
\end{remark}

\begin{remark}
In the case of Hochschild cohomology, the deformation functor of a morphism $R\xra{f} S$ of associative algebras can be defined over the category of Artinian associative algebras, rather than just $\C_{\L}$. This means that  the Lie bracket $H^i(f) \by \H^j(f) \to \H^{i+j+1}(f)$ defined in \cite{ddt1} \S \ref{ddt1-adams} extends to an associative cup product. If $f=\id_R$ is an identity, then we know that the Lie bracket vanishes (since $\ddef_f$ is a loop space, by Proposition \ref{loopmor}), which is why the cup product becomes commutative. Of course, we also have the bracket   $H^i(\id_R) \by \H^j(\id_R) \to \H^{i+j}(\id_R)$ associated to  the deformation functor of the object $R$. 
\end{remark}

\appendix

\section{Comparison with \cite{paper2}}\label{ddtconsistentsdc}

Now assume that $\L=k$, a field of characteristic $0$. In \cite{ddt1}, an equivalence was given between the homotopy category of $\Z$-graded DGLAs and $\Ho(sc\Sp)$. Under the equivalences of Theorems \ref{allequiv} and \ref{schrep}, this equivalence sends a DGLA to its associated deformation functor in the sense of \cite{Man2} (by \cite{ddt1} Remark \ref{ddt1-manchar}).

However, in \cite{paper2}, a functor $\cE$ was constructed from $\N_0$-graded DGLAs to SDCs, and Definition \ref{ddefdef} then gives us an associated object of $sc\Sp$. The purpose of this appendix is to show that the two constructions are consistent with each other.

\subsection{DGLAs}\label{dgla}

\begin{definition}
A differential graded Lie algebra (DGLA) is a  graded vector space $L=\bigoplus_{i } L^i$ over $k$, equipped with operators $[,]:L \by L \ra L$ bilinear and $d:L \ra L$ linear,  satisfying:

\begin{enumerate}
\item $[L^i,L^j] \subset L^{i+j}.$

\item $[a,b]+(-1)^{\bar{a}\bar{b}}[b,a]=0$.

\item $(-1)^{\bar{c}\bar{a}}[a,[b,c]]+ (-1)^{\bar{a}\bar{b}}[b,[c,a]]+ (-1)^{\bar{b}\bar{c}}[c,[a,b]]=0$.

\item $d(L^i) \subset L^{i+1}$.

\item $d \circ d =0.$

\item $d[a,b] = [da,b] +(-1)^{\bar{a}}[a,db]$
\end{enumerate}

Here $\bar{a}$ denotes the degree of $a$, mod $ 2$, for $a$ homogeneous.

A DGLA is said to be nilpotent if the lower central series $\Gamma_nL$ (given by $\Gamma_1L=L$, $\Gamma_{n+1}L= [L, \Gamma_nL]$) vanishes for $n\gg 0$.
\end{definition}

\begin{definition}
Given a nilpotent Lie algebra $\g$, define $\hat{\cU}(\g)$ to be the pro-unipotent completion of the universal enveloping algebra of $\g$, regarded as a pro-object in the category of algebras. As in \cite{QRat} Appendix A, this is a pro-Hopf algebra, and we define $\exp(\g) \subset  \hat{\cU}(\g)$ to consist of elements $g$ with $\vareps(g)=1$ and $\Delta(g)= g\ten g$, for $\vareps: \hat{\cU}(\g) \to k$ the augmentation (sending $\g$ to $0$), and $\Delta: \hat{\cU}(\g) \to \hat{\cU}(\g)\ten \hat{\cU}(\g)$ the comultiplication.

Since $k$ is assumed to have characteristic $0$, exponentiation gives an isomorphism from $\g$ to $\exp(\g)$, so we may regard $\exp(\g)$ as having the same elements as $\g$, but with multiplication given by the Campbell--Baker--Hausdorff formula. 
\end{definition}

\begin{definition}\label{mcldef}
Given a  nilpotent DGLA $L^{\bt}$, define the Maurer-Cartan set by 
$$
\mc(L):= \{\omega \in  L^{1}\ \,|\, d\omega + \half[\omega,\omega]=0 \in  \bigoplus_n L^{2}\}
$$
Define the gauge group $\Gg(L)$ by $\Gg(L):= \exp(L^0)$, which acts on $\mc(L)$ by the gauge action 
 $$
g(\omega)=   g\cdot \omega \cdot g^{-1} -dg\cdot g^{-1},
$$
where $\cdot$ denotes multiplication in the universal enveloping algebra of $L$. 
That $g(\omega) \in \mc(L)$ is a standard calculation (see \cite{Kon} or \cite{Man}).
\end{definition}

\begin{definition}
A morphism $f:L \to M$ of DGLAs is said to be a quasi-isomorphism if $\H^*(f): \H^*(L) \to \H^*(M)$ is an isomorphism.
\end{definition}

\begin{proposition}\label{mcallequiv}
There is a model structure on the category of $\Z$-graded DGLAs, in which weak equivalences are quasi-isomorphisms, and fibrations are surjections. This category is Quillen-equivalent to the model category $sDG\Sp$  of Definition \ref{sdgspdef}. 
\end{proposition}
\begin{proof}
\cite{ddt1} Lemma \ref{ddt1-dglamodel} and Corollary \ref{ddt1-mcallequiv}.
\end{proof}

\subsection{Cosimplicial groups}

\begin{definition}
Given an $\N_0$-graded DGLA $L$, let $DL$ be its denormalisation. This becomes a cosimplicial Lie algebra via the Eilenberg-Zilber shuffle product. Explicitly:

$$
D^nL:= \bigoplus_{\begin{smallmatrix} m+s=n \\ 1 \le j_1 < \ldots < j_s \le n \end{smallmatrix}} \pd^{j_s}\ldots\pd^{j_1}L^m,
$$
where    we define the $\pd^j$ and $\sigma^i$ using the simplicial identities, subject to the conditions that $\sigma^i L =0$ and $\pd^0v= dv -\sum_{i=1}^{n+1}(-1)^i \pd^i v$ for all $v \in L^n$.

We now have to define the Lie bracket $\llbracket -, -\rrbracket$ from $D^nL \ten D^nL$ to $D^n L$. Given a finite set  $I$ of strictly positive integers, write $\pd^I= \pd^{i_s}\ldots\pd^{i_1}$, for $I=\{i_1, \ldots i_s\}$, with $1 \le i_1 < \ldots < i_s$. The Lie bracket is then   defined on the basis by 
$$
\llbracket \pd^Iv, \pd^J w\rrbracket:= \left\{ \begin{matrix} \pd^{I\cap J}(-1)^{(J\backslash I, I \backslash J)}[v,w] & |v|= |J\backslash I|, |v|= |I\backslash J|,\\ 0 & \text{ otherwise},\end{matrix} \right.
$$
where for disjoint sets $S,T$ of integers, $(-1)^{(S,T)}$ is the sign of the shuffle permutation of $S \sqcup T $ which sends the first $|S|$ elements to $S$ (in order), and the remaining $|T|$ elements to $T$ (in order). 
Note that this description only works for $0 \notin I \cup J$. 
\end{definition}

\begin{definition}
Now recall from \cite{paper2} \S \ref{paper2-explicit}, that the functor $\cE:DG\LA \to \mathrm{SDC}$ from $\N_0$-graded DGLAs to SDCs is defined by
$$
\cE(L)^n(A)= \exp(D^n(L)\ten \m_A),
$$
making $\cE(L)$ into a cosimplicial complex of group-valued functors. To make it an SDC, we must define a $*$ product. We do this as the Alexander-Whitney cup product
$$
g*h = (\pd^{m+n}\ldots \pd^{m+2}\pd^{m+1}g)\cdot (\pd^0)^m h,
$$
 for $g \in \cE(L)^m,\,h \in \cE(L)^n$.
\end{definition}

\begin{definition}\label{mcdefgp}
Given a cosimplicial simplicial group $G$, define $\mmc(G) \in \bS$ by $\mmc(G) \subset \prod_{n\ge 0} (G^{n+1})^{\Delta^n}$, satisfying the conditions of \cite{htpy} Lemma \ref{htpy-maurercartan}, i.e.
 the elements $\omega_n \in(G^{n+1})^{\Delta^n}$ satisfy 
\begin{eqnarray*}
\pd_i\omega_n &=& \left\{\begin{matrix} \pd^{i+1}\omega_{n-1}  & i>0 \\ (\pd^1\omega_{n-1})\cdot(\pd^0\omega_{n-1})^{-1} & i=0,\end{matrix} \right.\\
\sigma_i\omega_n &=& \sigma^{i+1}\omega_{n+1},\\
\sigma^0\omega_n&=& 1.
\end{eqnarray*}
Define $\mc :sc\Gp \to \Set$ by $\mc(G)=\mmc(G)_0$.

There is an adjoint action  of $G^0$ on $\mmc(G)$, given by
$$
(g*\omega)_n= (\pd_0 (\pd^1)^{n+1}(\sigma_0)^{n+1}g) \cdot \omega_n \cdot (\pd^0 (\pd^1)^n(\sigma_0)^ng^{-1}),
$$
as in \cite{htpy} Definition \ref{htpy-defdef}.

We then define  $\ddel(G)$ to be the homotopy quotient $\ddel(G)= [\mmc(G)/G^0]\in \bS$.
\end{definition}

Let  $\exp$ denote exponentiation of a nilpotent Lie algebra (giving a unipotent group).

\begin{corollary}\label{sdcgp}
Given an $\N_0$-graded DGLA $L$, the deformation functor $\ddef(\cE(L)) \in sc\Sp$ 
is weakly equivalent to the functor 
$$
A \mapsto \ddel(\exp(DL\ten \m(A))).
$$
\end{corollary}
\begin{proof}
The SDC $\cE(L)$ corresponds to $A \mapsto \cE(\exp(DL \ten \m(A))$ in the notation of \cite{monad} \S \ref{monad-comonoid}, so the result is an immediate consequence of \cite{monad} Proposition \ref{monad-cfdef}.  
\end{proof}

\begin{corollary}
For  $L$ as above,  $\ddef(\cE(L)) $ is weakly equivalent in the model category $sc\Sp$ to the functor  $A \mapsto [\mc(\exp(DL\ten \m(A)))/\exp(L^0\ten \m(A_0))]$.
\end{corollary}
\begin{proof}
\cite{ddt1} Lemma \ref{ddt1-levelwiseqs} implies  that $ \mc(\exp(DL\ten \m(A)))\to \mmc(\exp(DL\ten \m(A)))$ defines a weak equivalence in $sc\Sp$    (although the former is not fibrant), and similarly for $L^0 \ten \m(A_0) \to L^0 \ten \m(A)$, so we get a weak equivalence on passing to the homotopy quotient.
\end{proof}

\subsection{The final comparison}

\begin{definition} 
Given an $\N_0$-graded DGLA $L$, define $\ddel(L) \in sDG\Sp$ to be the functor $\ddel(L) :dg\C_k \to \bS$ given by the homotopy quotient 
$$
A \mapsto [\mc(\Tot^{\Pi}(L\ten N\m(A)))/ \exp(L^0\ten \m(A_0)]
$$
with respect to the gauge action of Definition \ref{mcldef}.
\end{definition}

\begin{corollary}\label{gplie}
Given an $\N_0$-graded DGLA $L$, the deformation functor $\ddef(\cE(L)) \in sc\Sp$ is weakly equivalent  to $\oR\Spf N^* \ddel(L)$, for $\oR\Spf N^*$ as in Theorem \ref{nequiv}.
\end{corollary}
\begin{proof}
By Corollary \ref{sdcgp}, it suffices to show that the functors $\ddel(L)$
and 
$
A \mapsto \ddel(\exp(DL\ten \m(A)))$ are weakly equivalent in $sc\Sp$. It follows from \cite{ddt1} Lemma \ref{ddt1-levelwiseqs} that the latter is weakly equivalent to 
$$
A \mapsto [\mc(\exp(DL\ten \m(A)))/ \exp(L^0\ten \m(A_0))], 
$$
which is not fibrant in general. Now,  \cite{monad} Theorem \ref{monad-cfexp} implies that this is isomorphic to 
$$
A \mapsto [\mc(\Tot^{\Pi} (L\ten N\m(A)))/ \exp(L^0\ten \m(A_0)],
$$ 
which is just $\Spf N^* \ddel(L)$. \cite{ddt1} Lemma \ref{ddt1-underlinecoho} then implies that $\oR\Spf N^* F \cong\Spf N^* F$ for all levelwise quasi-smooth functors $F$.
\end{proof}

\begin{definition}
Define $DGdg\C_{k}$ to be the category of Artinian local  $\N_0\by\N_0$-graded  graded-commutative  $\L$-algebras $A^{\bt}_{\bt}$ with differential of bidegree $(1,-1)$ and  residue field $k$. Let   $dgDG\Sp$ be the category of left-exact $\Set$-valued functors on  $DGdg\C_{k}$.
\end{definition}

\begin{proposition}\label{finalequiv}
Under the equivalence of Proposition \ref{mcallequiv}, an $\N_0$-graded DGLA $L$ corresponds to the deformation functor $\ddef(\cE(L)) \in sc\Sp$ of Definition \ref{ddefdef}.
\end{proposition}
\begin{proof}
The equivalence of \cite{ddt1} Corollary \ref{ddt1-mcallequiv} is given by a functor $\oR\Spf D^* : \Ho(sDG\Sp) \to \Ho(dgDG\Sp)$, together with a functor
$$
\Spf \Tot^* \mc: DG_{\Z}\LA \to dgDG\Sp, 
$$
on $\Z$-graded DGLAs, given by 
$$
\Spf \Tot^* \mc(L)(A)= \mc( \Tot L\ten \m(A)). 
$$

By Corollary \ref{gplie}, it suffices to show that the objects $\oR \Spf D^* \ddel(L)$ and $\Spf \Tot^* \mc(L)$ are weakly equivalent in $dgDG\Sp$. 

Taking $A\in DGdg\C_{k}$, it follows from the definitions  that $\Spf D^* \ddel(L)(A)$ consists of maps $\Spf (DA) \to \mc(L)\by^{\exp(L^0)}W(\exp(L^0))$ in $sDG\Sp$, where $DA \in (dg\C_k)^{\Delta}$ is defined by cosimplicial denormalisation, and $\Spf(DA) \in sDG\Sp$ is the functor $dg\C_k \to \bS$ given in level $n$ by $\Hom_{dg\C_k}(D^nA, -)$.

Thus
$$
\Spf D^* \ddel(L)(A) \subset \mc(\Tot L\ten A^0)\by\mc(D(\exp(L^0\ten A_0^{\bt}))  
$$ 
consists of pairs $(\omega,g)$ with 
$
g* \pd^0_{DA} \omega = \pd^1_{DA}\omega,
$
corresponding in level $n$ to the map 
\begin{eqnarray*}
(\Spf D^n A) &\to& \mc(L)\by \exp(L^0)^n\\
 (\omega, g) & \mapsto& (\omega, (\pd^2)^{n-1}g, \pd^0(\pd^2)^{n-2}g, \ldots, (\pd^0)^{n-1}g).
\end{eqnarray*}

By \cite{monad} Theorem \ref{monad-cfexp}, we know that $\mc(D(\exp(L^0\ten A_0^{\bt})) \cong \mc( L^0 \ten A_0^{\bt}))$, giving us $\gamma \in \mc( L^0 \ten A_0^{\bt}))$. Note that $g=\exp(\gamma)$, so 
$$
g* \pd^0_{DA} \omega = \pd^0_{DA} \omega +[\gamma, \omega],
$$
with all higher terms   vanishing, since $\sigma^0 \gamma=0$, so $\llbracket \gamma, \llbracket \gamma, v\rrbracket \rrbracket=0$ for all $v$ (and in particular when $v= \pd^0_{DA} \omega$).

Thus  we have
$$
\Spf D^* \ddel(L)(A) = \{(\omega,\gamma) \,:\, \omega \in \mc(\Tot L\ten A^0),\, \gamma \in \mc(L^0\ten A_0^{\bt}),\, [\gamma, \omega] + d_{c,A} \omega=0 \},
$$ 
where $d_{c,A}$ is the cochain differential on $A$

Now look at $\gamma + \omega \in (\Tot(L\ten A))^1$. The equations combine to show that $\alpha:=\gamma + \omega$ lies in $\mc(\Tot(L\ten A))$, since
$$
d\alpha+[\alpha, \alpha] = (d_L\omega + d_A^s\omega +[\omega, \omega]) + (d_A\gamma + [\gamma, \gamma]) + ([\gamma, \omega] + d_{A,c} \omega)=0,
$$
where $d^s_A$ is the chain differential on $A$. 

Thus we have defined a map 
$$
\psi:\Spf D^* \ddel(L)\to \Spf (\Tot^{\Pi})^*\mc(L).
$$
In fact, we have shown that 
$$
\Spf D^* \ddel(L)(A) \cong \mc\Tot ((L\ten A^0)\by_{(L^0\ten A^0)} (L^0\ten A_0)),
$$
and \cite{ddt1} Lemma \ref{ddt1-dgweakext} then implies that $\psi$ is a weak equivalence (with similar reasoning to \cite{ddt1} Lemma \ref{ddt1-levelwiseqs}).

Finally, note that this gives cohomology groups (as defined in Definition \ref{generalcoho})  $\H^n(\Spf D^* \ddel(L)(A) )\cong \H^{n+1}(L)$, and that
$$
\H^n(\oR\Spf D^* \ddel(L))\cong \H^n(\ddel(L)) \cong \H^{n+1}(L)
$$
since  the equivalence of \cite{ddt1} Proposition \ref{ddt1-mcallequiv} preserves cohomology groups.
Therefore the morphism 
 $\Spf D^* \ddel(L)\to  \oR \Spf D^* \ddel(L)$ is also a weak equivalence by Corollary \ref{weak}, and this completes the proof.
\end{proof}

\bibliographystyle{alphanum}
\bibliography{references}
\end{document}